\documentclass[a4paper,11pt]{amsart}
\usepackage[utf8]{inputenc}
\usepackage{amsmath}
\usepackage{cases}
\usepackage{amsfonts}
\usepackage[colorlinks,linkcolor=blue,citecolor=blue]{hyperref}
\usepackage{latexsym, amssymb, amsmath, amsthm, bbm}
\usepackage[all]{xy}
\usepackage{pgfplots}

\DeclareSymbolFont{EulerExtension}{U}{euex}{m}{n}
\DeclareMathSymbol{\euintop}{\mathop} {EulerExtension}{"52}
\DeclareMathSymbol{\euointop}{\mathop} {EulerExtension}{"48}

\allowdisplaybreaks[4]

\def \id{\operatorname{Id}}

\def \C{\mathcal{C}}

\def \To{\longrightarrow}
\def \dim{\operatorname{dim}}

\def \Hom{\operatorname{Hom}}
\def \Ob{\operatorname{Ob}}
\def \End{\operatorname{end}}
\def \Coend{\operatorname{coend}}

\def \Id{\operatorname{Id}}

\def \id{\operatorname{Id}}

\def \C{\mathcal{C}}

\numberwithin{equation}{section}

\newtheorem{theorem}{Theorem}[section]
\newtheorem{lemma}[theorem]{Lemma}
\newtheorem{proposition}[theorem]{Proposition}

\newtheorem{definition}[theorem]{Definition}
\newtheorem{example}[theorem]{Example}

\begin{document}
\title{Construction of coend and the reconstruction theorem of bialgebras}
\thanks{$^\dag$Supported by NSFC 11722016.}


\author{Kun Zhou}
\address{Department of Mathematics, Nanjing University, Nanjing 210093, China} \email{dg1721021@smail.nju.edu.cn}
\date{}
\maketitle
\begin{abstract} Assume $k$ is a field and let $\mathcal{F}:\C\rightarrow Vect_{k}$ be a small $k$-linear functor from a $k$-linear abelian category $\C$ to the category of vector spaces over the field $k$, the purpose of this note is to use a little knowledge of linear algebra and category to give the description of $\End(\mathcal{F})$ and $\Coend(\mathcal{F})$, and then we give the reconstruction theorem of bialgebras by using this description. We use a constructive approach to understand $\End(\mathcal{F})$ and $\Coend(\mathcal{F})$. We must point out that this note does not draw any essentially new conclusions except for the fact that part of the proof is new, but rather to try to understand the concept of end (resp. coend) from a different perspective.
\end{abstract}

\section{Introduction}
Let $k$ be a field and let $\mathcal{F}:\C\rightarrow Vect_{k}$ be a $k$-linear functor from a small $k$-linear abelian category $\C$ to the category of vector spaces over the field $k$, and we will use the $\mathcal{F}$ and the $Vect_{k}$ without explanation in the following content. This paper is organized as follows. In Section 2, we will give the description of $\End(\mathcal{F})$ (resp. $\Coend(\mathcal{F})$) as vector spaces, and we will give an example of $\Coend(\mathcal{F})$ as our motivation for this note. Then we will give the coalgebra structure of $\Coend(\mathcal{F})$ when $\dim(\mathcal{F}(X))<\infty$ for all $X\in \Ob(\C)$, and then we will prove that $\End(\mathcal{F})=\Coend(\mathcal{F})^*$ as algebra by using a different way from that in \cite[Section 1.10]{PT}. If $(\C,\otimes,I)$ is a strict tensor category and $(\mathcal{F},\Id_I,\mathcal{F}_2)$ is a tensor functor such that $\textrm{Im}\mathcal{F}$ are finite dimensional, then we will use a constructive approach to give the bialgebra structure of $\Coend(\mathcal{F})$ in Section 3. At last, we give the reconstruction theorem of bialgebras by using our description of the bialgebra $\Coend(\mathcal{F})$. We must point out that this theorem has been obtained by many articles, what we have done is to describe the bialgebra structure of $\Coend(\mathcal{F})$ concretely.

\section{$\End(\mathcal{F})$ and $\Coend(\mathcal{F})$ as vector spaces}
Let $A$ be an algebra over $k$ and let $\mathcal{G}:A$-Mod$\To Vect_k$ be the forgetful functor from the category of left $A$ modules to $Vect_k$, then we can use the the algebra $\textrm{End}(\mathcal{G})=\text{Nat}(\mathcal{G}, \mathcal{G} )$ of natural transformations from $\mathcal{G}$ to $\mathcal{G}$ to reconstruct $A$ through the algebra isomorphism $\varphi:A\rightarrow \textrm{End}(\mathcal{G})$, where $\varphi(a)_{\textrm{V}}:=a_{\textrm{V}}$ for $\textrm{V}\in \textrm{Ob}(A\textrm{-Mod})$. Dually, if $C$ is a coalgebra over $k$ and let $\mathcal{E}:C$-Comod$\To Vect_k$ be the forgetful functor from the category of right $C$ comodules to $Vect_k$, then we can ask if we can reconstruct $C$ by using the functor $\mathcal{E}$ ? To answer this question, we need the dual object of $\textrm{End}(\mathcal{F})$ to reconstruct $C$. But the definition of $\textrm{End}(\mathcal{F})$ depends on elements of it, and thus we recall the equivalent definition of $\textrm{End}(\mathcal{F})$ by using morphisms to get the dual concept of it.
\subsection{Definition of $\End(\mathcal{F})$ and $\Coend(\mathcal{F})$ as vector spaces}
For convenience, we introduce a $k$-linear bifunctor $\Hom:\C^{op}\times \C \To Vect_k$, which is defined by $\Hom(X,Y):=\Hom_k(\mathcal{F}(X),\mathcal{F}(Y))$ for $X,Y\in \Ob(\C)$ and $\Hom(f^{op},g)(T):=F(g)\circ T \circ F(f)$ for $f\in \Hom_{\C}(X',X),\; g\in \Hom_{\C}(Y,Y'),\; T\in \Hom_k(\mathcal{F}(X),\mathcal{F}(Y))$, here $f^{op}$ is the morphism $f$ in $\C^{op}$. In particular, $\Hom(f^{op},\Id_Y)(T)=T \circ F(f)$ and $\Hom(\id_X,g)(T)=F(g)\circ T$. Let $\mathcal{H}:\C\rightarrow \mathcal{D}$ be a $k$-linear functor, where $\C,\;\mathcal{D}$ are $k$-linear abelian categories, then the general definition of $\End(\mathcal{H})$ and $\Coend(\mathcal{H})$ can be found in \cite[2.1.6. Definition]{P}. But for our purposes, we consider only the case where $\{\mathcal{H}(X)\}_{X\in \Ob(\C)}$ are vector spaces in this note.
\begin{definition}\cite[2.1.6. Definition]{P}\label{def2.1}
The $\End(\mathcal{F})$ is a vector space such that following conditions
\begin{itemize}
  \item[(i)] there is a family of linear maps $\{\pi_X:\End(\mathcal{F})\To \Hom(X,X)\}_{X\in \Ob(\C)}$,
  \item[(ii)] the following diagrams commute for all $X,Y\in \Ob(C)$ and $f\in \Hom_{\C}(X,Y)$,
  $$\xymatrix{
    \Hom(Y,Y)\ar[rr]^{\Hom(f^{op},\Id_Y)} & & \Hom(X,Y) \\
    \End(\mathcal{F})\ar[u]^{\pi_{Y}}\ar[rr]^{\pi_{X}} & & \Hom(X,X)\ar[u]_{\Hom(\Id_X,f)}
     }$$
    Here $\pi_X$ and $\pi_Y$ are linear maps of \emph{(i)}.
  \item[(iii)] $\End(\mathcal{F})$ is a final object of $Vect_k$ such that \emph{(i)} and \emph{(ii)}, i.e if there is a vector space $U$ and there are linear maps $\{\pi_X^U:U \To \Hom(X,X)\}_{X\in \Ob(\C)}$ making the following diagrams commute for all $X,Y\in \Ob(\C)$ and $f\in \Hom_{\C}(X,Y)$,
      $$\xymatrix{
    \Hom(Y,Y)\ar[rr]^{\Hom(f^{op},\Id_Y)} & & \Hom(X,Y) \\
    \End(\mathcal{F})\ar[u]^{\pi_{Y}^U}\ar[rr]^{\pi_{X}^U} & & \Hom(X,X)\ar[u]_{\Hom(\Id_X,f)}
     }$$
      then we can find a unique linear map $\varphi:U\To \End(\mathcal{F})$ making the following diagrams commute for all $X\in \Ob(\C)$,
      $$\xymatrix{
    U \ar[rr]^-{\varphi}\ar[dr]_-{\pi_X^U} & & \End(\mathcal{F}) \ar[dl]^-{\pi_X} \\
     & \Hom(X,X) &
     }$$

\end{itemize}
\end{definition}
It can be seen that the above condition (ii) is equivalent to $\pi_Y(x)\circ \mathcal{F}(f)=\mathcal{F}(f)\circ \pi_X(x)$ for $X,Y\in \Ob(\C),\;f\in \Hom_{\C}(X,Y),\;x\in \End(\mathcal{F})$. If $U$ is a vector space with a family of linear maps $\{\pi_X^U:U \To \Hom(X,X)\}_{X\in \Ob(\C)}$ and these linear maps satisfy (ii), (iii) of Definition \ref{def2.1}, then we will call that $U$ is a realization of $\End(\mathcal{F})$. Naturally we have following proposition
\begin{proposition}\label{pro2.2}
The $\End(\mathcal{F})$ above exists and it is unique up to isomorphism.
\end{proposition}
\begin{proof}
If we define $\{\pi_X^{\mathcal{F}} :\textrm{End}(\mathcal{F})\To \Hom(X,X)\}_{X\in \Ob(\C)}$ by $\pi_X^{\mathcal{F}}(\eta)=\eta_{X}$ where $\eta\in \textrm{End}(\mathcal{F})$ and ${X\in \Ob(\C)}$, then it is easy to see that these linear maps satisfy the conditions (ii),(iii) of Definition \ref{def2.1}. Therefore $\textrm{End}(\mathcal{F})$ is a realization of $\End(\mathcal{F})$. Because $\End(\mathcal{F})$ is a final object, we know $\End(\mathcal{F})$ is unique up to isomorphism as vector space.
\end{proof}
Dually, we can now define the $\Coend(\mathcal{F})$ by changing the direction of arrows of diagrams in Definition \ref{def2.1}.
\begin{definition}\cite[2.1.6. Definition]{P}\label{def2.3}
The $\Coend(\mathcal{F})$ is a vector space such that following conditions
\begin{itemize}
  \item[(i)] there is a family of linear maps $\{i_X:\Hom(X,X)\To \Coend(\mathcal{F})\}_{X\in \Ob(\C)}$,
  \item[(ii)]  the following diagrams commute for all $X,Y\in \Ob(C)$ and $f\in \Hom_{\C}(Y,X)$,
  $$\xymatrix{
    \Hom(Y,Y)\ar[d]_{i_Y} & & \Hom(X,Y)\ar[ll]_{\Hom(f^{op},\Id_Y)} \ar[d]^{\Hom(\Id_X,f)}\\
    \Coend(\mathcal{F}) & & \Hom(X,X)\ar[ll]_{i_X}
     }$$
     Here $i_X$ and $i_Y$ are linear maps of \emph{(i)}.
  \item[(iii)] $\Coend(\mathcal{F})$ is an initial object of $Vect_k$ such that \emph{(i)} and \emph{(ii)}, i.e if there is a vector space $U$ and there are linear maps $\{i_X^U:\Hom(X,X) \To  U \}_{X\in \Ob(\C)}$ making the following diagrams commute for all $X,Y\in \Ob(C)$ and $f\in \Hom_{\C}(Y,X)$,
      $$\xymatrix{
    \Hom(Y,Y)\ar[d]_{i_Y^U} & & \Hom(X,Y)\ar[ll]_{\Hom(f^{op},\Id_Y)} \ar[d]^{\Hom(\Id_X,f)}\\
    \Coend(\mathcal{F}) & & \Hom(X,X)\ar[ll]_{i_X^U}
     }$$
       then we can find a unique linear map $\varphi:\Coend(\mathcal{F})\To U$ making the following diagrams commute for all $X\in \Ob(\C)$,
      $$\xymatrix{
    \Coend(\mathcal{F}) \ar[rr]^-{\varphi} & & U  \\
     & \Hom(X,X)\ar[ul]^-{i_X} \ar[ur]_-{i_X^U} &
     }$$

\end{itemize}
\end{definition}
If $U$ is a vector space with a family of linear maps $\{i_X^U:\Hom(X,X) \To  U \}_{X\in \Ob(\C)}$ which satisfy the conditions (ii), (iii) of Definition \ref{def2.3}, then we will call $U$ is a realization of $\Coend(\mathcal{F})$. Note that the above condition (ii) is equivalent to $i_X(\mathcal{F}(f)\circ T)=i_Y(T\circ\mathcal{F}(f))$ for $X,Y\in \Ob(\C),\;f\in \Hom_{\C}(Y,X)$ and $T\in \Hom(X,Y)$, which we will use frequently.
Similar to Proposition \ref{pro2.2}, we have
\begin{proposition}\label{pro2.4}
The $\Coend(\mathcal{F})$ above exists and it is unique up to isomorphism.
\end{proposition}
\begin{proof}
Let $V:=\bigoplus_{X\in \Ob(\C)}\Hom(X,X)$ and we denote $J$ be the subspace of $V$ which is linear spanned by $\{F(f)\circ T- T\circ F(f)|\;X,Y\in \Ob(\C),f\in \Hom_{\C}(X,Y),T\in \Hom(Y,X)\}$. If we define $\{j_X: \Hom(X,X)\To  V/J \}_{X\in \Ob(\C)}$ by $j_X(S)=S+J$ for $X\in \Ob(\C)$, then it is easy to see that these linear maps satisfy the conditions (ii),(iii) of Definition \ref{def2.1}. Therefore the quotient space $V/J$ is a realization of $\Coend(\mathcal{F})$. Since $\Coend(\mathcal{F})$ is an intial object, we know $\Coend(\mathcal{F})$ is unique up to isomorphism as vector space.
\end{proof}
Returning to the original problem at the beginning of this section, that is can we reconstruct the coalgebra $C$ by using the forgetful functor $\mathcal{E}:C$-Comod$\To Vect_k$ ? For technical reasons, we use another forgetful functor $\mathcal{T}:C$-Comod$_f\To Vect_k^f$ instead of using the $\mathcal{E}$, where $C$-Comod$_f$ is the category of finite dimensional right $C$ comodules and $Vect_k^f$ is the category of finite dimensional vector spaces over $k$. The following example shows that the coalgebra $C$ can be reconstructed from $\Coend(\mathcal{T})$ as vector space. Furthermore, we will prove that the coalgebra $C$ can be reconstructed from $\Coend(\mathcal{T})$ as coalgebra in Example \ref{ex2.8} and this example is the main motivation of this note. Recall that if $U,V$ are finite dimensional vector spaces and if we denote the linear dual space of $U$ by $U^*$, then $\Hom_k(U,V)=\{\beta\otimes x|\;\beta\in U^*,x\in V, (\beta\otimes x)(y):=\beta(y)x\}$. In order to describe the following results more convenient, we will use this fact often without further comment. Moreover we agree that $\beta\otimes x,\gamma\otimes y...$ means that $(\beta\otimes x)(u):=\beta(u)x,(\gamma\otimes y)(u):=\gamma(u)y...$, where English letters represent elements of a given vector space and Greek alphabet represent elements of dual space of this vector space.

\begin{example}\label{ex2.5}
\emph{Let $\mathcal{T}:C$-Comod$_f \To Vect_k^f$ be the forgetful functor above, and if we define a family of linear maps $\{i_X:\Hom_k(X,X)\To C\}_{X\in \Ob(\C)}$ by $i_X(\beta\otimes x):=\beta(x_{(0)})x_{(1)}$ then we will show that $C$ is a realization of $\Coend{\mathcal{T}}$. Therefore the coalgebra $C$ can be reconstructed from $\Coend(\mathcal{T})$ as vector space.
\\Let $g\in \Hom_{\C}(X,Y)$ and taking $(\gamma\otimes y) \in \Hom_k(Y,X)$, then we have $g(y_{(0)})\otimes y_{(1)}=g(y)_{(0)}\otimes g(y)_{(1)}$ by definition of the $g$. So $i_Y[(\gamma \otimes y)\circ g]=i_X[g\circ (\gamma \otimes y)]$ and this implies that $\{i_X\}_{X\in \Ob(\C)}$ such that (ii) of Definition \ref{def2.3}.
\\Assume $U$ is a vector space with linear maps $\{i_X^U:\Hom_k(X,X)\To U\}_{X\in \Ob(\C)}$ and these maps satisfy (ii) of Definition \ref{def2.3}. Define $\varphi:C\To U$ by $\varphi(c):=i_{N_c}^U(\epsilon\otimes c)$ for $c\in C$, where $N_c$ is the smallest subcoalgebra of $C$ which contains $c$ and its comodule structure is given by $\Delta$. Then we will show $\varphi$ satisfy the diagram of (iii) in Definition \ref{def2.3} and it is unique.
\\Taking $(\beta\otimes x)\in \Hom_k(X,X)$, directly we have $\varphi\circ i_X(\beta\otimes x)=i_{N_c}^U(\epsilon\otimes c)$ where $ c=\beta(x_{(0)}) x_{(1)}$. Since $X$ is finite dimensional, we can choose a finite dimensional subcoalgebra $C_1\subseteq C$ such that $\rho_X(X)\subseteq X\otimes C_1$. Then $\rho_X\in \Hom_{\C}(X,X\otimes C_1)$ where $\rho_{X\otimes C_1}:=(\Id\otimes \Delta)$. Let $(\beta\otimes \epsilon)\otimes x\in \Hom_k(X\otimes C_1,X)$ defined by $((\beta\otimes \epsilon)\otimes x)(y\otimes d):=\beta(y)\epsilon(d)x$, then $((\beta\otimes \epsilon)\otimes x)\circ \rho_X=\beta\otimes x$ and $\rho_X\circ ((\beta\otimes \epsilon)\otimes x)=(\beta\otimes \epsilon)\otimes (x_{(0)}\otimes x_{(1)})$. Since (ii) of Definition \ref{def2.3}, we get
\begin{align}
\label{e.2.1}i_X^U(\beta\otimes x)=i_{X\otimes C_1}^U[(\beta\otimes \epsilon)\otimes (x_{(0)}\otimes x_{(1)})].
\end{align}
Let $f:=(\beta\otimes \Id)\in \Hom_{\C}(X\otimes C_1,C_1)$ and $T:=[\epsilon\otimes (x_{(0)}\otimes x_{(1)})]\in \Hom_k(C_1,X\otimes C_1)$, where $(\beta\otimes \Id)(x\otimes c_1):=\beta(x)c_1$ and $[\epsilon\otimes (x_{(0)}\otimes x_{(1)})](c_1):=\epsilon(c_1)x_{(0)}\otimes x_{(1)})$ for $x\in X,c_1\in C$, then we know $i_{X\otimes C_1}^U(T\circ f)=i_{C_1}^U(f\circ T)$ since (ii) of Definition \ref{def2.3}. Because $T\circ f=(\beta\otimes \epsilon)\otimes (x_{(0)}\otimes x_{(1)})$ and $f\circ T=\epsilon\otimes c$, we get
\begin{align}
\label{e.2.2}i_{X\otimes C_1}^U[(\beta\otimes \epsilon)\otimes (x_{(0)}\otimes x_{(1)})]=i_{C_1}^U(\epsilon\otimes c).
\end{align}
Owing to \eqref{e.2.1}, \eqref{e.2.2}, we know $i_X^U(\beta\otimes x)=i_{C_1}^U(\epsilon\otimes c)$. Because there is a natural inclusion $N_c\hookrightarrow C_1$, we get $i_{C_1}^U(\epsilon\otimes c)=i_{N_c}^U(\epsilon\otimes c)$. Since $\varphi\circ i_X(\beta\otimes x)=i_{N_c}^U(\epsilon\otimes c)$ and $i_X^U(\beta\otimes x)=i_{C_1}^U(\epsilon\otimes c)$, we have $\varphi\circ i_X(\beta\otimes x)=i_X^U(\beta\otimes x)$ and hence $\varphi\circ i_X=i_X^U$. Note that $\varphi\circ i_{N_c}(\epsilon\otimes c)=i_X^U(\epsilon\otimes c)$ which gives the uniqueness of $\varphi$.}
\end{example}

\subsection{Coalgebra structure on $\Coend(\mathcal{F})$}
We have known that the $\textrm{End}(\mathcal{F})$ is a realization of $\End(\mathcal{F})$ and because the $\textrm{End}(\mathcal{F})$ is an algebra, so $\End(\mathcal{F})$ has a natural algebra structure. At the same time, the algebra structure of $\End(\mathcal{F})$ can also be determined by the following way. Assume $\{\pi_X^U:U \To \Hom(X,X)\}_{X\in \Ob(\C)}$ is a realization of $\End(\mathcal{F})$, if we consider linear maps $\{m_X:\End(\mathcal{F})\otimes \End(\mathcal{F})\To \Hom(X,X)\}_{X\in \Ob(\C)},\;\{\eta_X:k\To \Hom(X,X)\}_{X\in \Ob(\C)}$, where $m_X,\eta_X$ are defined by $m_X(a\otimes b):=\pi_X(a)\pi_X(b)$ and $\eta_X(1):=\Id_X$ for $X\in \Ob(\C)$, then we will get the multiplication and the unit of $\End(\mathcal{F})$ through these linear maps by using (iii) of Definition \ref{def2.1}, i.e the the multiplication $m$ and the unit $\eta$ of $\End(\mathcal{F})$ are the unique linear maps making the following diagrams commute for all $X\in \Ob(\C)$,
      $$\xymatrix{
    \End(\mathcal{F})\otimes \End(\mathcal{F}) \ar[rr]^-{m}\ar[dr]_-{m_X} & & \End(\mathcal{F}) \ar[dl]^-{\pi_X} \\
     & \Hom(X,X) &
     }
     \xymatrix{
     k \ar[rr]^-{\eta}\ar[dr]_-{\eta_X} & & \End(\mathcal{F}) \ar[dl]^-{\pi_X} \\
     & \Hom(X,X) &
     }
     $$
Dually, we can define coalgebra structure of $\Coend(\mathcal{F})$. For technical reasons, we require $\mathcal{F}(X)$ are finite dimensional for all $X\in \Ob(\C)$ in the following content. Let $\{x_i\}_{i=1}^{n}$ be a basis for $\mathcal{F}(X)$ and let $\{x^i\}_{i=1}^{n}$ be the dual basis for $\mathcal{F}(X)^*$, and we recall that $\Hom(X,X)$ has a natural coalgebra structure which is defined by $\Delta_X(\beta\otimes x):=\sum_{i=1}^n (\beta\otimes x_i)\otimes (x^i\otimes x)$ and $\epsilon_X(\beta\otimes x):=\beta(x)$, where $\beta\in X^*,x\in X$ and $(\beta\otimes x)(y):=\beta(y)x,(x^i\otimes x)(y):=x^i(y)x$. Now we can use this family of coalgebra algebras $\{\Hom(X,X)\}_{X\in \Ob(\C)}$ to give coalgebra structure of $\Coend(\mathcal{F})$.

\begin{definition}\label{def2.6}
Assume that $\{i_X^U:\Hom(X,X)\To U\}_{X\in \Ob(\C)}$ is a realization of $\Coend(\mathcal{F})$ and let $\{\Delta_X^U:\Hom(X,X)\To U\otimes U\}_{X\in \Ob(\C)},\;\{\epsilon_X^U:\Hom(X,X)\To k\}_{X\in \Ob(\C)}$ be linear maps defined by $\Delta_X^U:=(i_X^U\otimes i_X^U)\circ \Delta_X$ and $\epsilon_X^U:=\epsilon_X$. Then the coalgebra structure $(U,\Delta_U,\epsilon_U)$ is defined by the unique linear maps $\Delta_U,\epsilon_U$ which making the following diagrams commute for all $X\in \Ob(\C)$,
$$\xymatrix{
    U \ar[rr]^{\Delta_U}& & U\otimes U \\
     & \Hom(X,X) \ar[ul]^-{i_X^U} \ar[ur]_-{\Delta_X^U}&
     }
    \xymatrix{
    U \ar[rr]^{\epsilon_U}& & k \\
     & \Hom(X,X) \ar[ul]^-{i_X^U} \ar[ur]_-{\epsilon_X^U}&
     }
$$
\end{definition}
The following two propositions show that existence of $\Delta_U,\epsilon_U$ of Definition \ref{def2.6} and $(U,\Delta_U,\epsilon_U)\cong (V,\Delta_V,\epsilon_V)$, where $\{i_X^U:\Hom(X,X)\To U\}_{X\in \Ob(\C)}$ is another realization of $\Coend(\mathcal{F})$ and $(V,\Delta_V,\epsilon_V)$ is given by the Definition \ref{def2.6} above. For these reasons, the coalgebra $\Coend(\mathcal{F})$ given by the Definition \ref{def2.6} is unique up to isomorphism as coalgebra.
\begin{proposition}\label{pro2.7}
The $\Delta_U,\epsilon_U$ of Definition \ref{def2.6} exist and $(U,\Delta_U,\epsilon_U)$ is a coalgebra.
\end{proposition}
\begin{proof}
To show the $\Delta_U,\epsilon_U$ exist, it is enough to check that $\{\Delta_X^U:\Hom(X,X)\To U\otimes U\}_{X\in \Ob(\C)},\;\{\epsilon_X^U:\Hom(X,X)\To k\}_{X\in \Ob(\C)}$ satisfy the (ii) of Definition \ref{def2.3}. Let $\{x_i\}_{i=1}^{n}$ (resp. $\{y_i\}_{i=1}^{m}$) be a basis for $\mathcal{F}(X)$ (resp. $\mathcal{F}(Y)$) and let $\{x^i\}_{i=1}^{n}$ (resp. $\{y^i\}_{i=1}^{m}$) be the dual basis for $\mathcal{F}(X)^*$ (resp. $\mathcal{F}(Y)^*$), and taking $f\in \Hom_{\C}(X,Y)$, $(\beta\otimes x)\in \Hom(Y,X)$, then
we have
\begin{align*}
\Delta_X^U((\beta\otimes x)\circ \mathcal{F}(f))&=\sum_{i=1}^n i_X^U(\beta\circ \mathcal{F}(f)\otimes x_i)\otimes i_X^U(x^i\otimes x)\\
                &=\sum_{i=1}^n i_Y^U(\beta \otimes \mathcal{F}(f)(x_i))\otimes i_X^U(x^i\otimes x)\\
        &=\sum_{i}^n \sum_{j=1}^mi_Y^U(\beta \otimes y_j <y^j,\mathcal{F}(f)(x_i)>)\otimes i_X^U(x^i\otimes x)\\
        &=\sum_{i}^n \sum_{j=1}^m i_Y^U(\beta \otimes y_j)\otimes i_X^U(<y^j,\mathcal{F}(f)(x_i)> x^i\otimes x)\\
        &=\sum_{j=1}^m i_Y^U(\beta \otimes y_j)\otimes i_X^U(y^j\circ\mathcal{F}(f)\otimes x)\\
        &=\sum_{j=1}^m i_Y^U(\beta \otimes y_j)\otimes i_X^U(y^j\otimes \mathcal{F}(f)(x))\\
        &=\Delta_Y^U(\mathcal{F}(f)\circ (\beta\otimes x) )
\end{align*}
and $\epsilon_X^U((\beta\otimes x)\circ \mathcal{F}(f))=\beta\circ\mathcal{F}(f)(x)=\epsilon_Y^U(\mathcal{F}(f)\circ (\beta\otimes x) )$, hence we know the $\Delta_U,\epsilon_U$ exist. Let $T\in \Hom(X,X)$ and we denote $i_X^U(T)$ of Definition \ref{def2.6} by $\overline{T}$, then $\Delta_U(\overline{T})=(i_X^U\otimes i_X^U)\circ \Delta_X(T)$ and $\epsilon_U(\overline{T})=\epsilon_X(T)$ by definition and so $(\Delta_U \otimes \Id)\circ \Delta_U(\overline{T})=(\Id \otimes \Delta_U )\circ \Delta_U(\overline{T}), \;(\epsilon_U \otimes \Id)\circ \Delta_U(\overline{T})=(\Id \otimes \epsilon_U )\circ \Delta_U(\overline{T})=(\overline{T})$. Since Proposition \ref{pro2.4}, we know $U$ is linear spanned by $\{\overline{T}|\;T\in \Hom(X,X)\}_{X\in \Ob(\C)}$ and hence $(U,\Delta_U,\epsilon_U)$ is a coalgebra.
\end{proof}

\begin{proposition}\label{pro2.8}
$(U,\Delta_U,\epsilon_U)\cong (V,\Delta_V,\epsilon_V)$ as coalgebra.
\end{proposition}
\begin{proof}
Because $\Coend(\mathcal{F})$ is an initial object of $Vect_k$, there exists invertible linear maps $\varphi:U\rightarrow V$ such that $\varphi(i_X^U(T))=i_X^V(T)$. Then we will check that $\varphi$ is a coalgebra map. Directly we have $\Delta_V(\varphi(i_X^U(T)))=(i_X^V\otimes i_X^V)\circ \Delta_X(T)=(\varphi\otimes \varphi)\circ \Delta_U(i_X^U(T))$ and $\epsilon_U(i_X^U(T))=\epsilon_X(T)=\epsilon_V(\varphi(i_X^U(T)))$. Note that $U$ is linear spanned by $\{i_X^U(T)|\;T\in \Hom(X,X)\}_{X\in \Ob(\C)}$ and hence $\varphi$ is a coalgebra isomorphism.
\end{proof}

We already saw in Example \ref{ex2.5} that $\Coend(\mathcal{T})$ can reconstruct $C$ as vector space, furthermore, the following example shows that $\Coend(\mathcal{T})$ also reconstruct $C$ as coalgebra.

\begin{example}\label{ex2.8}
\emph{Let $\mathcal{T}:\C\To Vect_k^f$ be the forgetful functor of Example \ref{ex2.5}, and we have known that $C$ is a realization of $\Coend{\mathcal{T}}$ as vector space. Therefore we have a coalgebra structure on $C$ by Definition \ref{def2.6} and we denote it by $(\Delta_C,\epsilon_C)$. Now let's prove that $(\Delta_C,\epsilon_C)=(\Delta,\epsilon)$, where $(\Delta,\epsilon)$ is the coalgebra structure of $C$ itself. We still use the notation of Example \ref{ex2.5} in the following context. Let $c\in C$ and we choose a finite dimensional subcoalgebra $C_1$ of $C$ which contains the $c$, then $i_{C_1}(\epsilon\otimes c)=c$ by definition. For convenience, we denote $i_{X}(T)$ by $\overline{T}$ and hence $c=\overline{\epsilon\otimes c}$. Let $\{x_i\}_{i=1}^{n}$ be a basis for $C_1$ and let $\{x^i\}_{i=1}^{n}$ be the dual basis for $C_1^*$, then we have
\begin{align*}
\Delta_C(\overline{\epsilon\otimes c})&=\overline{(\epsilon\otimes c)_{(1)}} \otimes \overline{(\epsilon\otimes c)_{(2)}}\\
                &=\sum_{i=1}^n \overline{(\epsilon\otimes x_i)}\otimes \overline{(x^i\otimes c)}\\
        &=\sum_{i=1}^n x_i \otimes <x^i,c_{(1)}>c_{(2)}\\
        &=c_{(1)}\otimes c_{(2)}=\Delta(c).
\end{align*}
Due to $\epsilon_C(\overline{\epsilon\otimes c})=\epsilon_{C_1}(\epsilon\otimes c)=\epsilon(c)$, we get $\epsilon_C=\epsilon$ and hence $(\Delta_C,\epsilon_C)=(\Delta,\epsilon)$. That is to say the coalgebra $\Coend{\mathcal{T}}$ reconstruct the coalgebra $C$.}
\end{example}
Since $\Coend(\mathcal{F})$ and $\End(\mathcal{F})$ are dual concept, we have the following proposition and we give its proof by using a different way from that in \cite[Section 1.10]{PT}.

\begin{proposition}\cite[Section 1.10]{PT}\label{pro2.9}
$\End(\mathcal{F})\cong \Coend(\mathcal{F})^*$ as algebra.
\end{proposition}
\begin{proof}
Assume $\{i_X:\Hom(X,X)\To \Coend(\mathcal{F})\}_{X\in \Ob(\C)}$ is a realization of $\Coend(\mathcal{F})$, and let $\{\pi_X:\Coend(\mathcal{F})^*\To \Hom(X,X)\}_{X\in \Ob(\C)}$ be a family of linear maps which are defined by $\pi_X(\alpha):=\sum_{i,j=1}^n \alpha\circ i_X(x^i\otimes x_j)x^j\otimes x_i$ for $X\in \Ob(\C)$, where $\{x^i\}_{i=1}^{n}$ is the dual basis of $\mathcal{F}(X)^*$ corresponding with a given basis $\{x_i\}_{i=1}^{n}$ of $\mathcal{F}(X)$ and $(x^i\otimes x_j)(x):=x^i(x)x_j$, then we can see that $\{\pi_X\}_{X\in \Ob(\C)}$ are algebra maps due to $\{i_X\}_{X\in \Ob(\C)}$ are coalgebra maps. To complete the proof, we only need to show that $\{\pi_X:\Coend(\mathcal{F})^*\To \Hom(X,X)\}_{X\in \Ob(\C)}$ is a realization of $\End(\mathcal{F})$.

Firstly, we show that $\{\pi_X\}_{X\in \Ob(\C)}$ satisfy (ii) of Definition \ref{def2.1}. Let $f\in \Hom_{\C}(X,Y)$, $\alpha\in \Coend(\mathcal{F})^*$, and we assume $\{y^i\}_{i=1}^{m}$ is the dual basis of $\mathcal{F}(Y)^*$ corresponding with a given basis $\{y_i\}_{i=1}^{m}$ of $\mathcal{F}(Y)$. Since
\begin{align*}
\mathcal{F}(f)\circ \pi_X(\alpha)&=\sum_{i,j=1}^n \alpha\circ i_X(x^i\otimes x_j)x^j\otimes \mathcal{F}(f)(x_i)\\
                &=\sum_{i,j=1}^n \sum_{k=1}^m\alpha\circ i_X(x^i\otimes x_j)x^j\otimes y_k <y^k, \mathcal{F}(f)(x_i)> \\
        &=\sum_{i=1}^n \sum_{k=1}^m \alpha\circ i_X(y^k\circ \mathcal{F}(f) \otimes x_j)x^j\otimes y_k\\
        &=\sum_{i=1}^n \sum_{k=1}^m \alpha\circ i_X(y^k \otimes \mathcal{F}(f) (x_j))x^j\otimes y_k\\
        &=\sum_{i=1}^n \sum_{j,k=1}^m \alpha\circ i_X(y^k \otimes <y^i,\mathcal{F}(f) (x_j)> y_i)x^j\otimes y_k\\
        &=\sum_{i,k=1}^m \alpha\circ i_X(y^k \otimes y_i)y^i\circ \mathcal{F}(f) \otimes y_k\\
        &=\pi_Y\circ \mathcal{F}(f)(\alpha),
\end{align*}
we know $\{\pi_X\}_{X\in \Ob(\C)}$ satisfy (ii) of Definition \ref{def2.1}. \\
Secondly, we prove that $\{\pi_X\}_{X\in \Ob(\C)}$ such that (iii) of Definition \ref{def2.1}. If there is a vector space $U$ with linear maps $\{\pi_X^U:U \To \Hom(X,X)\}_{X\in \Ob(\C)}$ and we assume that these maps satisfy (ii) of Definition \ref{def2.1}, then $\{i_X^U:\Hom(X,X)\To U^*\}_{X\in \Ob(\C)}$ which are defined by $i_X^U(\beta\otimes x)(u):=\beta\circ \pi_X^U(u)(x)$ are linear maps satisfy (ii) of Definition \ref{def2.3}. Therefore there is a unique linear map $\varphi:\Coend(\mathcal{F})\To U^*$ such that $\varphi\circ i_X=i_X^U$. Define $\varphi^*:U\To \Coend(\mathcal{F})^*$ by $\varphi^*(u)(c)=\varphi(c)(u)$ for $u\in U$ and $c\in \Coend(\mathcal{F})$, then we know $\pi_X\circ \varphi^*=\pi_X^U$ due to $\varphi\circ i_X=i_X^U$ and hence the following diagrams commute for all $X\in \Ob(\C)$.
      $$\xymatrix{
    U \ar[rr]^-{\varphi^*}\ar[dr]_-{\pi_X^U} & & \Coend(\mathcal{F})^* \ar[dl]^-{\pi_X} \\
     & \Hom(X,X) &
     }$$
Let's prove that a linear map satisfying the above commutative diagrams is unique. Assume that there is another linear map $\phi:U\To \Coend(\mathcal{F})^*$ such that $\pi_X\circ \phi=\pi_X^U$, then $\pi_X(\phi(u))=\pi_X(\varphi^*(u))$ since we have shown $\pi_X\circ \varphi^*=\pi_X^U$. But $\pi_X(\alpha)=\sum_{i,j=1}^n \alpha\circ i_X(x^i\otimes x_j)x^j\otimes x_i$ for $\alpha\in \Coend(\mathcal{F})^*$ by definition, so $\phi(u)\circ i_X=\varphi^*(u)\circ i_X$ for $X\in \Ob(\C)$. This implies $\phi(u)=\varphi^*(u)$ and hence $\varphi^*$ is unique if it satisfies the above diagrams.
\end{proof}

\section{Reconstruction theorem of bialgebras}\label{sec2.2}
We assume that $(\C,\otimes,I)$ is a strict tensor category and $(\mathcal{F},\Id_I,\mathcal{F}_2)$ is a tensor functor from $\C$ to $Vect_k^f$ in this section. Recall that if $\C$ is only a $k$-linear abelian category and $\mathcal{F}$ is a $k$-linear functor then we can only get coalgebra structure of $\Coend(\mathcal{F})$, but now we've added a tensor structure to $\C$, naturally we can think of adding some information to $\Coend(\mathcal{F})$. In order to express this idea precisely, we need the following lemmas. Define $\mathcal{F}\otimes \mathcal{F}:\C\times \C\To Vect_k^f$ by $(\mathcal{F}\otimes \mathcal{F})(X,Y)=\mathcal{F}(X)\otimes\mathcal{F}(Y),\;(\mathcal{F}\otimes \mathcal{F})(f,g)=\mathcal{F}(f)\otimes\mathcal{F}(g)$ for $X,Y\in \Ob(\C)$ and $f\in \Hom_{\C}(X,X'),\;g\in \Hom_{\C}(Y,Y')$, then we have the following result. The proof of it is different from that in \cite[Section 8]{A}.

\begin{lemma}\cite[Section 8]{A}\label{lem3.1}
$\Coend(\mathcal{F}\otimes \mathcal{F})\cong \Coend(\mathcal{F})\otimes \Coend(\mathcal{F})$ as coalgebra, where $\Coend(\mathcal{F})\otimes \Coend(\mathcal{F})$ has the tensor product coalgebra structure.
\end{lemma}
\begin{proof}
Assume $\{i_X:\Hom(X,X)\To \Coend(\mathcal{F})\}_{X\in \Ob(\C)}$ is a realization of $\Coend(\mathcal{F})$, and we define $i_{X,Y}:\Hom_k(\mathcal{F}(X)\otimes\mathcal{F}(Y),\mathcal{F}(X)\otimes\mathcal{F}(Y))\To \Coend(\mathcal{F})\otimes \Coend(\mathcal{F})$ by $i_{X,Y}(S\otimes T)=i_X(S)\otimes i_Y(T)$ for $S\in \Hom(X,X),\;T\in \Hom(\mathcal{F}(Y),\mathcal{F}(Y))$. It can be seen that $\Coend(\mathcal{F})\otimes \Coend(\mathcal{F})$ with these maps is a realization of $\Coend(\mathcal{F}\otimes \mathcal{F})$ and $\{i_{X,Y}\}_{X,Y\in \Ob(\C)}$ are coalgebra maps, so $\Coend(\mathcal{F}\otimes \mathcal{F})\cong \Coend(\mathcal{F})\otimes \Coend(\mathcal{F})$ as coalgebra.
\end{proof}

Let $\{v_i\}_{i=1}^{n}$ be a basis for $V$ and let $\{v^i\}_{i=1}^{n}$ be the dual basis for $V^*$, and we recall that the comatrix coalgebra $\Hom_k(V,V)$ is defined by $\Delta(\beta\otimes v):=\sum_{i=1}^n (\beta\otimes v_i)\otimes (v^i\otimes v)$ and $\epsilon(\beta\otimes v):=\beta(v)$, where $(\beta\otimes v)(w):=\beta(w)v$ for $\beta\in V^*,\;v,w\in V$, then we have
\begin{lemma}\label{lem3.2}
Assume $P\in \Hom_k(V,V)$ and it is invertible, then $\varphi_P:\Hom_k(V,V) \To \Hom_k(V,V)$ is a coalgebra automorphism, where $\varphi_P(T):=PTP^{-1}$.
\end{lemma}
\begin{proof}
Let $\{v_i\}_{i=1}^{n}$ be a basis for $V$, and we consider the non-degenerate dual pair $<,>:\Hom_k(V,V)\times \Hom_k(V,V)\To k$ which is defined by $<\beta \otimes v, \gamma\otimes w>=\beta(w)\gamma(v)$, where $(\beta \otimes v)(y):=\beta(y)v$ and $ (\gamma\otimes w)(y):=\gamma(y)w$. Then we can define $\phi_P:\Hom_k(V,V)\times \Hom_k(V,V)\To k$ through the equation $<\phi_P(T),S>=<T,\varphi_P(S)>$, and we can check that $\phi_P(T)=P^tT(P^t)^{-1}$ where $P^t$ is the transpose matrix of $P$ for the given basis $\{v_i\}_{i=1}^n$ of $V$. Therefore $\phi_P$ is an algebra map, which implies $\varphi_P$ is a coalgebra map.
\end{proof}

Given a realization of $\Coend(\mathcal{F})$ by $\{i_X:\Hom(X,X)\To \Coend(\mathcal{F})\}_{X\in \Ob(\C)}$, and we denote $i_X(T)$ by $\overline{T}$ for $T\in \Hom(X,X)$. Due to Lemma \ref{lem3.1}, we can use the following commute diagrams for all $X,Y\in \Ob(\C)$ to define multiplication of $\Coend(\mathcal{F})$
$$\xymatrix{
    \Hom_k(\mathcal{F}(X)\otimes\mathcal{F}(Y),\mathcal{F}(X)\otimes\mathcal{F}(Y)) \ar[rr]^-{m_{X,Y}} \ar[d]_-{\overline{\mathcal{F}_2}} & & \Coend(\mathcal{F})\\
    \Hom_k(\mathcal{F}(X \otimes Y),\mathcal{F}(X \otimes Y)) \ar[urr]_-{i_{X\otimes Y}} &  &
     },$$
where $\overline{\mathcal{F}_2}(S\otimes T):=\mathcal{F}_2(S\otimes T)\mathcal{F}_2^{-1}$ for $S\in \Hom(X,X),\;T\in \Hom(Y,Y)$. Then the above commute diagrams determine the product of $\Coend(\mathcal{F})$ by $m(\overline{S}\otimes \overline{T})=\overline{\mathcal{F}_2(S\otimes T)\mathcal{F}_2^{-1}}$. Moveover if we define unit $\eta$ of $\Coend(\mathcal{F})$ by $\eta:=i_I$, then we have

\begin{proposition}\label{pro3.2}
$(\Coend(\mathcal{F}),\Delta,\epsilon,m,\eta)$ is a bialgebra.
\end{proposition}
\begin{proof}
Firstly, we show $(\Coend(\mathcal{F}),m,\eta)$ is an algebra. For simple, we denote the identity map from $\Coend(\mathcal{F})$ to $\Coend(\mathcal{F})$ by $id$. By definition, $m\circ(m\otimes id)(\overline{W})=i_{(X\otimes Y)\otimes Z}[\mathcal{F}_2\circ (\mathcal{F}_2\otimes id) \circ W \circ (\mathcal{F}_2\otimes id)^{-1}\circ \mathcal{F}_2^{-1}]$ and $m\circ(id_X \otimes m)(\overline{W})=i_{X\otimes (Y\otimes Z)}[\mathcal{F}_2\circ (id \otimes \mathcal{F}_2) \circ W \circ (id \otimes \mathcal{F}_2)^{-1}\circ \mathcal{F}_2^{-1}]$ where $W\in \Hom_k(\mathcal{F} (X)^{\otimes 3},\mathcal{F}(X)^{\otimes 3})$. Since $\C$ is strict and $[\mathcal{F}_2\circ (\mathcal{F}_2\otimes id)=\mathcal{F}_2\circ (id \otimes \mathcal{F}_2)$ due to $(\mathcal{F},\Id_I,\mathcal{F}_2)$ is a tensor functor, we know $m$ satisfy associativity. Due to $m\circ (id\otimes \eta)(\overline{S})=i_{X}[(\mathcal{F}_2)_{X,I}(S\otimes 1)(\mathcal{F}_2^{-1})_{X,I}]=\overline{S}$, we have $m\circ (id\otimes \eta)=id$. Similarly, we get $(id\otimes \eta)\circ m=id$ and hence $(\Coend(\mathcal{F}),m,\eta)$ is an algebra.

Secondly, we show $m,\eta$ are coalgebra maps. It can be seen that $\eta$ is a coalgebra map by definition of coalgebra structure of $\Coend(\mathcal{F})$. To prove $m$ is a coalgebra map, we need only to check that $\{m_{X,Y}\}_{X,Y\in \Ob(\C)}$ are coalgebra maps. We have known $i_{X\otimes Y}$ is a coalgebra map since the definition of coalgebra structure of $\Coend(\mathcal{F})$.  Because $m_{X,Y}=i_{X\otimes Y}\circ \overline{\mathcal{F}_2}$ and $\overline{\mathcal{F}_2}$ is a coalgebra map by Lemma \ref{lem3.2}, we get that $m_{X,Y}$ is a coalgebra map and hence we have completed the proof.
\end{proof}
Since $\Coend(\mathcal{F})$ is a bialgebra, we know that finite dimensional right comodules of $\Coend(\mathcal{F})$ is a strict tensor category and we denote it by $\C_{\mathcal{F}}$. Then the tensor functor $\mathcal{F}$ induces a unique tensor functor $\mathcal{E}:\C \To \C_{\mathcal{F}}$ satisfies $\mathcal{F}=\mathcal{F}'\circ \mathcal{E}$ and $\mathcal{F}_2=\mathcal{F}'_2\circ \mathcal{E}_2$, where $\mathcal{F}':\C_{\mathcal{F}} \To Vect_k^f$ is the forgetful functor. A natural question is when $\mathcal{E}$ is an equivalent tensor functor? The following reconstruction theorem of bialgebras is the answer to this question.
\begin{theorem}\cite[Theorem 3]{A}
If $\mathcal{F}$ is exact and faithful, then $\mathcal{E}$ is an equivalence of tensor functor.
\end{theorem}

\end{document}